\documentclass[a4paper,12pt,reqno]{amsart}
 \pdfoutput=1 
\usepackage{amsmath,amsthm,amssymb}
\usepackage{hyperref,cleveref}
\usepackage{mathrsfs}
\usepackage{graphicx}
\usepackage{comment}

\allowdisplaybreaks[1]

\title{Baker domain for a transcendental skew product}
\newtheorem{thm}{Theorem}
\newtheorem{lemma}{Lemma}

\newtheorem*{remark}{Remark}
\author[R. Kaur]{Ramanpreet Kaur }
\address{Department of Mathematics, University of Delhi,
Delhi--110 007, India}

\email{preetmaan444@gmail.com}

\begin{document}
\title[Baker domain for a transcendental skew-product]{Baker domain for a transcendental skew-product}
\begin{abstract}
\begin{comment}
In this note, we demonstrate the construction of a domain for the existence of Baker domain of a transcendental skew-product. In addition, we use plurisubharmonic method to show that the constructed domain infact is an absorbing domain for the same.
\end{comment}
In this note, we demonstrate the existence of a Baker domain of a transcendental skew-product by constructing a domain that is shown to be absorbing using plurisubharmonic method.
\end{abstract}
\maketitle
\section{Introduction and Preliminaries}
This note is aimed at showing that the transcendental skew-product  $F:\mathbb{C}^2\to\mathbb{C}^2$ given by
\begin{equation}\label{1}
F(z,w)=(e^{-(z+w)}+z+w, e^{-2w}+2w+1)
\end{equation} has a Baker domain.
Firstly, recall that a polynomial skew-product on $\mathbb{C}^2$ is of the following form
\[F(z,w)=(p(z,w),q(w)), \text{ for }(z,w)\in\mathbb{C}^2.\]
Here, $q$ is polynomial in one variable and $p$ is a  polynomial in two variables. On similar lines, we define a transcendental skew-product, a mapping $F:\mathbb{C}^2\to\mathbb{C}^2$ of the form
$F(z,w)=(f(z,w), g(w))$, where $f(z,w)$ and $g(w)$ are transcendental entire functions. The Fatou set of $F$, denoted by $\mathscr{F}(F)$, is  the largest open set where the family of iterates $\{F^n\}_{n\in\mathbb{N}}$ of $F$ is normal \cite{beardon}. A connected component of the Fatou set is called a Fatou component. The complement of the Fatou set is called the Julia set. For a detailed discussion on the Fatou--Julia theory in several variables, see \cite{Fornaess,Fornaess2,Rudin} and the references therein.

For a rational function of degree $d\geq 2$ in one variable, every Fatou component is pre-periodic \cite{Sullivan}. In particular, the periodic Fatou component, say $U$, takes one of the forms: attracting, parabolic, Siegel disk and Herman ring. 
\begin{comment}
\begin{enumerate}
\item Attracting: every point of the component converges to a periodic point $z_0\in U$ under iteration of the function. In this case, the multiplier at the periodic point has a modulus less than 1.
\item Parabolic: every point of the component converges to a periodic point $z_0\in \partial U$ under iteration of the function. In this case, the multiplier at the periodic point is $e^{\iota a},$ for some $ a\in\mathbb{Q}$ and hence has a modulus equal to 1.
\item Siegel disk: the function on $U$ is conformally conjugate to an irrational rotation of a unit disk. In this case, the multiplier at the periodic point  $e^{\iota a},$ for some $ a\in \mathbb{R}\setminus \mathbb{Q}$ and hence has a modulus equal to 1.
\item Herman ring: the function on $U$ is conformally conjugate to an irrational rotation of a standard annulus.
\end{enumerate} 
\end{comment}
In contrast to this, a transcendental entire function does not have Herman rings. Indeed, it may have:
\begin{enumerate}
\item Baker domain: a Fatou component $U$ for which the orbit of every point of $U$ converges to $\infty$, an essential singularity \cite{Fatou}.
\item Wandering domain: a Fatou component $U$ that is not pre-periodic \cite{Baker1}. 
\end{enumerate}
These variations have led to the study of the classification of Fatou components in higher variables. Several authors have discussed about these components for maps of certain types, for example, recurrent Fatou components of a transcendental H\'enon map \cite{Arosio}, and Fatou components of attracting polynomial skew-products \cite{Peters}. 

In this note, we shall restrict ourselves to Baker domain corresponding to $F$ mentioned in \Cref{1}. The motivation of such a map comes from \cite{Arosio}, where the authors give an example of a transcendental H\'enon map having Baker domain. They used the plurisubharmonic method to show that for a particular  transcendental H\'enon map with Baker domain, there exists an absorbing domain. We shall also apply the same method to show that $F$ mentioned in \Cref{1} has 
\begin{itemize}
\item a Baker domain.
\item there exists an absorbing domain for such a Baker domain.
\end{itemize} 
In essence, we prove the following:
\begin{thm}\label{T1}
There exists a transcendental skew-product having Baker domain.
\end{thm}

To demonstrate Baker domain of a transcendental skew-product, first we need the definition of essential singularity at infinity for the same. We also define an essential singularity at infinity of a map $F:\mathbb{C}^2\to\mathbb{C}^2$ as it is defined for transcendental H\'enon map \cite{Arosio}, that is, $[p:q:0]$ is an essential singularity at infinity if for any point $[s:t:0]$ there exists a sequence $(z_n,w_n)\in\mathbb{C}^2$ converging to $[s:t:0]$ such that $F(z_n,w_n)$ converges to  $[p:q:0]$.

Our first observation is that $[2:1:0]$ is an essential singularity for $F$. To see this, take 
\[h(\zeta)=\frac{e^{-3\zeta}+3\zeta-1}{4\zeta}.\]
It can be easily seen that $h$ is a transcendental entire function. Clearly, there exists a sequence $\{\zeta_n\}$ converging to infinity such that $h(\zeta_n)\to\frac{s}{t}$. Now, we have 
\begin{align*}
F(\zeta_n, 2\zeta_n)&=(e^{-3\zeta_n}+3\zeta_n, e^{-4\zeta_n}+4\zeta_n+1)\\
&=(h(\zeta_n)4\zeta_n+1, e^{-4\zeta_n}+4\zeta_n+1)\\
&\to[s:t:0].
\end{align*}
\section{The construction of Baker domain}
As a first step towards Baker domain, let us first define an invariant domain for $F$. For each $\alpha>1$, let
\[L_{\alpha}:=\{(z,w)\in\mathbb{C}^2: \operatorname{Re}w>\operatorname{Re}z+\alpha, \operatorname{Re}z>1 , \operatorname{Re}w> 1\}.\]
 We shall first observe that $L_{\alpha}$ is an invariant set for each $\alpha>1$. Let $(z_0,w_0)\in L_{\alpha}$, that is, 
$\operatorname{Re}w_0> \operatorname{Re}z_0+\alpha$. We have to show that $F(z_0,w_0)=(z_1,w_1)\in L_{\alpha}$. Consider 
\begin{align*}
\operatorname{Re}w_1-\operatorname{Re}z_1&= \operatorname{Re}(e^{-2w_0}+2w_0+1)-\operatorname{Re}(e^{-(z_0+w_0)}+z_0+w_0)\\
&=\operatorname{Re}w_0-\operatorname{Re}z_0+1+\operatorname{Re}e^{-2w_0}-\operatorname{Re}e^{-(z_0+w_0)}\\
&\geq \alpha+1-e^{-2\operatorname{Re}w_0}-e^{-\operatorname{Re}(z_0+w_0)}\\
&\geq \alpha\quad \quad(\text{because }\operatorname{Re}z_0>1 \text{ and }\operatorname{Re}w_0>1)
\end{align*}
This shows that $F(z_0,w_0)=(z_1,w_1)\in L_{\alpha}$, that is, $L_{\alpha}$
is an invariant domain. 

In particular, 
\[L=\bigcup\limits_{\alpha>1}L_{\alpha}\]
is an invariant set.
\begin{remark}
If $L_{\alpha}$ is chosen as above, then  we have $\operatorname{Re}w_n>2 \operatorname{Re}w_0+\frac{n}{2}$ and $\operatorname{Re}z_n> \operatorname{Re}z_0+\frac{n}{2}$, for every $n\in\mathbb{N}$ (Note that $(z_n,w_n)=F^n(z_0,w_0)$, for every $n\in\mathbb{N}$). This, in particular implies that $|w_n|$ and $|z_n|$ converges to infinity uniformly on compact subsets of $L$.
\end{remark}
\begin{lemma}\label{l1}
The iterates of $F$ converges uniformly on compact subsets of $L$ to $[2:1:0]$.
\end{lemma}
\begin{proof}
Let $K$ be a compact subset of $L$, therefore, $K$ is contained in $L_{\alpha}$, for some $\alpha>0$. By induction we have 
\[w_n-z_n=w_0-z_0+n+\sum\limits_{i=0}^{n-1}e^{-2w_i}-\sum\limits_{i=0}^{n-1}e^{-(z_i+w_i)}.\]
Using the above remark and definition of $L_{\alpha}$, we have
 $\frac{w_n}{z_n}$ converges to $2$ uniformly on $K$.
\end{proof}
\begin{proof}[Proof of \Cref{T1}]
Using \Cref{l1}, $L$ is contained in an invariant Fatou component, say $U$. This, in particular tells us that $U$ is a Baker domain for $F$.
\end{proof}
 Now, we show that $L$ is an absorbing domain for $U$, that is, 
\[U=A:= \mathop{\cup}_{n\in\mathbb{N}}F^{-n}(L).\]

To prove this, we shall use the plurisubharmonic method \cite{Arosio}. Recall that a function $u:\Omega \subseteq_{\text{open}}\mathbb{C}^n\to\mathbb{R}$ is said to be plurisubharmonic if  
\begin{enumerate}
\item $u$ is upper semicontinuous  and is not identically $-\infty$ on any connected component of $\Omega$.
\item for each $a\in\Omega$ and $b\in\mathbb{C}^n$, the function $\lambda\mapsto u(a+\lambda b)$ is subharmonic or identically $-\infty$ on every component of the set 
$\{\lambda\in\mathbb{C}: a+\lambda b\in \Omega\}$.
\end{enumerate} 
In this case, we write $u\in\mathcal{PSH}(\Omega)$.
Now, we define a sequence of plurisubharmonic functions $u_n:U\to \mathbb{R}$ as 
\[u_n(z_0,w_0)=\frac{-(\operatorname{Re}w_n-\operatorname{Re}z_n)}{|w_n|+|z_n|}-1.\]
It can be easily seen that the functions $u_n$ are uniformly bounded from above by $0$, and hence $\limsup\limits_{n\to\infty}u_n\leq 0 $.

\begin{lemma}\label{l2}
If $H$ is a compact subset of $A$, then $\limsup\limits_{n\to\infty}u_n\leq -1 $ on $H$.
\end{lemma}
\begin{proof}
Let $(z_0,w_0)\in L_{\alpha}$, for some $\alpha>0$. Then, we have
$\operatorname{Re}w_n-\operatorname{Re}z_n>\alpha$, which means that 
\[\frac{-(\operatorname{Re}w_n-\operatorname{Re}z_n)}{|w_n|+|z_n|}<-\frac{\alpha}{|w_n|+|z_n|}.\] This further implies that 
\[\frac{-(\operatorname{Re}w_n-\operatorname{Re}z_n)}{|w_n|+|z_n|}-1<-\frac{\alpha}{|w_n|+|z_n|}-1.\] 
Hence, 
$\limsup\limits_{n\to\infty}u_n(z_0,w_0)\leq -1 $. Now, there exist a finite number of $\alpha$'s, say $\alpha_1, \alpha_2,\ldots, \alpha_k$, such that 
\[H\subset \bigcup\limits_{i=1}^{k}F^{-n_i}(L_{\alpha_{i}}).\]
Therefore, on $H$, we have $\limsup\limits_{n\to\infty}u_n\leq -1 $.
\end{proof}
\begin{lemma}\label{l3}
On $U\setminus A$, we have $ \limsup\limits_{n\to\infty}u_n=0$.
\end{lemma}
\begin{proof}
We shall prove it by contradiction. Suppose that for some $(z_0,w_0)\in U\setminus A$, there exists $\alpha_1>0$ and $N\in\mathbb{N}$ such that $u_n(z_0,w_0) <-\alpha_1$, for every $n\geq N$, that is, $\operatorname{Re}w_n-\operatorname{Re}z_n\geq (\alpha_1-1)(|w_n|+|z_n|)$, for every $n\geq \mathbb{N}$.  Therefore, we can choose sufficiently large $n$ such that $\operatorname{Re}w_n-\operatorname{Re}z_n> \alpha_2$, where $\alpha_2>1$ (since $|w_n| $ and $|z_n|$ tends to infinity as $n$ tends to infinity).
This, in particular means that $(z_n,w_n)\in L_{\alpha_2}$, which further implies that $(z_0,w_0)\in A$, which is a contradiction.
\end{proof}
\begin{lemma}\label{l4}
The set $L$ is an absorbing domain.
\end{lemma}
\begin{proof}
Suppose that $U\not=A$. Define $u(z)=\limsup\limits_{n\to\infty} u_n(z)$ and let $u^*$ be its upper semicontinuous regularization. Then by \cite[Proposition 2.9.17]{Klimek}, $u^*$ is plurisubharmonic. By \Cref{l2} and \Cref{l3}, the function $u^*$ is strictly negative on $A$ and identically equal to zero on $U\setminus A$, this contradicts the sub mean value property at boundary points of $A$.
\end{proof}
\section*{Acknowledgement}
I would like to thank my thesis advisor Professor Sanjay Kumar, Department of Mathematics, Deen Dayal Upadhyaya College, Delhi, for several fruitful discussions. 
The author is supported by the National Board for Higher Mathematics, India.

\end{document}